\documentclass[11pt,reqno]{amsart}

\newtheorem{proposition}{Proposition}[section]
\newtheorem{lemma}[proposition]{Lemma}
\newtheorem{corollary}[proposition]{Corollary}
\newtheorem{theorem}[proposition]{Theorem}

\theoremstyle{definition}
\newtheorem{definition}[proposition]{Definition}
\newtheorem{example}[proposition]{Example}
\newtheorem{examples}[proposition]{Examples}
\newtheorem{remark}[proposition]{Remark}
\newtheorem{remarks}[proposition]{Remarks}

\newcommand{\thlabel}[1]{\label{th:#1}}
\newcommand{\thref}[1]{Theorem~\ref{th:#1}}
\newcommand{\selabel}[1]{\label{se:#1}}
\newcommand{\seref}[1]{Section~\ref{se:#1}}
\newcommand{\lelabel}[1]{\label{le:#1}}
\newcommand{\leref}[1]{Lemma~\ref{le:#1}}
\newcommand{\prlabel}[1]{\label{pr:#1}}
\newcommand{\prref}[1]{Proposition~\ref{pr:#1}}
\newcommand{\colabel}[1]{\label{co:#1}}
\newcommand{\coref}[1]{Corollary~\ref{co:#1}}

\newcommand{\exlabel}[1]{\label{ex:#1}}
\newcommand{\exref}[1]{Example~\ref{ex:#1}}
\newcommand{\delabel}[1]{\label{de:#1}}
\newcommand{\deref}[1]{Definition~\ref{de:#1}}
\newcommand{\eqlabel}[1]{\label{eq:#1}}
\newcommand{\equref}[1]{(\ref{eq:#1})}

\newcommand{\Hom}{{\rm Hom}}

\def\ot{\otimes}

\newcommand{\Cc}{\mathcal{C}}
\newcommand{\Dd}{\mathcal{D}}

\newcommand{\Mm}{\mathcal{M}}

\def\*C{{}^*\hspace*{-1pt}{\Cc}}
\def\text#1{{\rm {\rm #1}}}

\def\ul{\underline}
\def\dul#1{\underline{\underline{#1}}}
\def\Nat{\dul{\rm Nat}}

\usepackage{amssymb}
\usepackage{color,amssymb,graphicx,amscd}

\input xy
\xyoption {all} \CompileMatrices

\begin{document}

\title[Four problems regarding
representable functors]{Four problems regarding representable
functors}

\author{G. Militaru}\thanks{This work was supported by
CNCSIS grant 24/28.09.07 of PN II "Groups, quantum groups, corings
and representation theory". }
\address{Faculty of Mathematics and Computer Science, University of Bucharest,
Str. Academiei 14, RO-010014 Bucharest 1, Romania}
\email{gigel.militaru@fmi.unibuc.ro and gigel.militaru@gmail.com}


\begin{abstract} Let $R$, $S$
be two rings, $C$ an $R$-coring and ${}_{R}^C{\mathcal M}$ the
category of left $C$-comodules. The category ${\bf Rep}\, (
{}_{R}^C{\mathcal M}, {}_{S}{\mathcal M} )$ of all representable
functors ${}_{R}^C{\mathcal M} \to {}_{S}{\mathcal M}$ is shown to
be equivalent to the opposite of the category ${}_{R}^C{\mathcal
M}_S$. For $U$ an $(S,R)$-bimodule we give necessary and
sufficient conditions for the induction functor $U\otimes_R - :
{}_{R}^C\mathcal{M} \to {}_{S}\mathcal{M}$ to be: a representable
functor, an equivalence of categories, a separable or a Frobenius
functor. The latter results generalize and unify the classical
theorems of Morita for categories of modules over rings and the
more recent theorems obtained by Brezinski, Caenepeel et al. for
categories of comodules over corings.
\end{abstract}
\date{}
\maketitle

\section*{Introduction}
Let $\mathcal{C}$ be a category and $\mathcal{V}$ a variety of
algebras in the sense of universal algebras. A functor $F:
\mathcal{C} \to \mathcal{V} $ is called representable
\cite{bergman} if $ \gamma \circ F : \mathcal{C} \to {\mathcal
Set}$ is representable in the classical sense, where $\gamma :
\mathcal{V} \to {\mathcal Set}$ is the forgetful functor. Four
general problems concerning representable functors have been
identified:

\textit{Problem A: Describe the category }${\bf Rep}\,
(\mathcal{C}, {\mathcal V})$\textit{ of all representable
functors} $F : \mathcal{C} \to \mathcal{V}$.

\textit{Problem B: Give a necessary and sufficient condition for a
given functor }$F : \mathcal{C} \to \mathcal{V}$ \textit{to be}
\textit{representable (possibly predefining the object of
representability)}.

\textit{Problem C: When is a composition of two representable
functors a representable functor?}

\textit{Problem D: Give a necessary and sufficient condition for a
representable functor $F : \Cc \to {\mathcal V}$ and for its left
adjoint to be separable or Frobenius.}

The pioneer of studying problem A was Kan \cite{kan} who described
all representable functors from semigroups to semigroups. A
crucial step related to problem A was made by Freyd in
\cite{freyd2}: if $\Cc$ is a cocomplete category and ${\mathcal
V}$ a variety of algebras then a functor $F: \Cc \to {\mathcal V}$
is representable if and only if $F$ is a right adjoint
(\cite[Theorem 8.14]{bergman}). A book dedicated exclusively to
problem A is \cite{bergman} where the category ${\bf Rep}\, (\Cc,
{\mathcal V})$ is described for different categories of varieties
of algebras $\Cc$ and ${\mathcal V}$. The fundamental example is
the following (\cite[Theorem 13.15]{bergman}): let $R$ be a ring
and $R-{\mathcal Rings}$ the category of $R$-rings. Then the
functor
$$
{\mathcal Y} : ({}_R\Mm _R)^{\rm op} \to {\bf Rep} \, (R-{\mathcal
Rings}, \, {\mathcal Ab}), \quad {\mathcal Y}(M):= {}_RHom_R (M,
-)
$$
is an equivalence of categories, where ${\mathcal Ab}$ is the
category of abelian groups. G. Janelidze pointed out that problem
A can be rephrased in a more elegant manner as follows: Let $T$ be
the corresponding Lawvere theory associated to $\mathcal{V}$. Then
a representable functor $F : \Cc \to \mathcal{V}$ is just a
functor $F : \Cc \to {\mathcal Set}$ equipped with an isomorphism
$F \cong \Hom_{\Cc} (C, -)$ and a $T$-coalgebra structure on $C$
(that is, a structure making $C$ a model of $T$ in $\Cc^{\rm
op}$). Then the problem A is reduced to: \textit{Describe the
category of models of $T$ in $\Cc^{\rm op}$.}

Concerning the problem B, several universal constructions in
mathematics like free groups, tensor products of modules, tensor
algebras, algebras of noncommutative differential forms give
answers to it in the trivial case $\mathcal{V} = {\mathcal Set}$.
We shall indicate two examples in the case of categories of
modules. For an $(S, R)$-bimodule $V$, the induction functor $
V\ot_R - : {}_R\Mm \to {}_S\Mm$ is representable if and only if
$V$ is finitely generated projective as a right $R$-module
\cite[Theorem 2.1]{morita}. On the other hand the property of a
functor to be Frobenius can be restated more elegantly as a
representability problem, predefining the object of
representability. For instance, \cite[Theorem 4.2]{CaenepeelMZ97b}
can be restated as follows: Let $H$ be a Hopf algebra over a field
and ${}_H ^H{\mathcal YD}$ be the category of Yetter-Drinfel'd
modules over $H$. Then the forgetful functor $F : {}_H ^H{\mathcal
YD} \to {}_H{\mathcal M}$ is representable having $H\otimes H$ as
a representing object if and only if $H$ is finite dimensional and
unimodular.

The problem C has a positive answer for categories of modules: the
tensor product of bimodules is responsible for this as
$$
{}_R{\rm Hom} (V, - ) \circ {}_S{\rm Hom} (U, - ) \cong {}_S{\rm
Hom} (U\otimes_R V,  - )
$$
if $R$, $S$, $T$ are rings, $U$ an $(S, R)$-bimodule and $V$ an
$(R, T)$-bimodule.

The problem D essentially depends on the nature of categories
$\Cc$ and ${\mathcal V}$. For example we can easily show that any
representable functor ${\rm Hom}_{{\mathcal Set}} (A, -) :
{\mathcal Set} \to {\mathcal Set}$ is separable while, if $\Cc =
{\mathcal Gr}^f$ is the category of finite groups, then no
representable functor $ Hom_{{\mathcal Gr}^{f}} (G, -) : {\mathcal
Gr}^f \to {\mathcal Set}$ is separable. Let now $U$ be an $(R,
S)$-bimodule and ${}^*U := {}_R\Hom (U, R) \in {}_{S}\Mm _{R}$.
Then the representable functor ${}_R\Hom (U, -) : {}_R\Mm \to
{}_S\Mm$ is separable if and only if there exists $\sum_i u_i
\ot_S u_i^* \in (U\ot_S {}^*U)^R$ such that $\sum_i u_i^* (u_i) =
1_R$ \cite[Corollary 5.8]{CaenepeelKadison}. The separability of
its left adjoint $U \otimes_S - $ was solved in \cite[Corollary
5.11]{CaenepeelKadison} in case $U$ is a finitely generated and
projective right $S$-module (in general this is still an open
problem).

In this paper we shall give answers to all the above problems in
case $\Cc = {}_{R}^C{\mathcal M}$, the category of left
$C$-comodules over an $R$-coring $C$ and $\mathcal{V} =
{}_{S}{\mathcal M}$, the category of left $S$-modules over a ring
$S$. For more details about the importance of corings and
comodules we refer to \cite{BrzezinskiWisbauer}. The paper is
organized as follows. In \seref{1} we recall the basic concepts
that will be used throughout the paper. In \seref{2} we prove all
technical results that we shall use to prove the main theorems of
the paper. We are focusing on the categories $\ul {Functors}\,
\bigl( {}_{S}{\mathcal M}, \, {}_{R}^C{\mathcal M} \bigl)$ and
$\ul {Functors}\, \bigl( {}_{R}^C{\mathcal M}, \, {}_{S}{\mathcal
M} \bigl)$ of all covariant functors that connect the category of
comodules over an $R$-coring $C$ and the category of modules over
a ring $S$. Two Yoneda type embeddings are constructed and the
classes of all natural transformations between an induction
functor and the identity functor on the category
${}_{R}^C\mathcal{M}$ are explicitly computed. \seref{3} contains
the main results of the paper. \thref{3.2} gives an answer for
Problem A: the category ${\bf Rep}\, ( {}_{R}^C{\mathcal M},
{}_{S}{\mathcal M} )$ is equivalent to the opposite of the
category ${}_{R}^C{\mathcal M}_S$. \coref{proC} offers an answer
for Problem C. Let $U$ be an $(S,R)$-bimodule: \thref{3.4} gives
necessary and sufficient conditions for the induction functor
$U\otimes_R - : {}_{R}^C\mathcal{M} \to {}_{S}\mathcal{M}$ to be a
representable functor, i.e. an answer for Problem B. It
generalizes and unifies two theorems that at first glance have
nothing in common: \cite[Theorem 2.1]{morita} is recovered for the
trivial coring $C = R$, while \cite[Theorem 4.1]{Brz:cor} is
obtained as a particular case for $U = S = R$ if in addition to
that we impose and predefine $C$ to be the object of
representability of $R\otimes_R -$. \exref{exem} and \coref{proC2}
explain that various theorems (\cite[Theorem 2.4]{CaenepeelMZ97b},
\cite[Theorem 4.1]{Brz:cor} etc.) giving necessary and sufficient
conditions for a forgetful functor to be Frobenius are particular
cases of representability. As a bonus of our approach, \thref{3.3}
gives necessary and sufficient conditions for $ U\otimes_R - $ to
be an equivalence of categories. Finally, \coref{3.5} and
\coref{3.7} give necessary and sufficient conditions for two types
of induction functors to be separable functors in case there
exists what we have called a comodule dual basis of first (or
second) kind: both are answers for Problem D.

\section{Preliminaries}\selabel{1}
We denote by ${\mathcal Set}$ the category of sets. All functors
in this paper will be covariant functors. $\Cc^{\rm op}$ will be
the opposite of a category $\Cc$. We denote by $\Nat (F, G)$ the
class of all natural transformations between two functors $F$, $G
: \Cc \to \Dd$ and by $ \ul {Functors}\, (\Cc, \Dd) = \Dd ^{\Cc}$
the category of all functors $F: \Cc\to \Dd$. The morphisms
between two functors $F$, $G \in \Dd ^{\Cc}$ are all natural
transformations $\varphi : F \to G$.

Let $R$, $S$ be two rings. We denote by ${}_R\Mm $, $\Mm _S$,
${}_{R}\Mm _{S}$ the categories of left $R$-modules, right
$S$-modules, $(R,S)$-bimodules. ${}_R\Hom (M, N)$, ${}\Hom_S (M,
N)$, ${}_R\Hom_S (M, N)$ will be the morphisms in the respective
categories. For an $R$-bimodule $M$ we denote by $M^R = \{ m \in M
\, | \, rm = mr, \, \forall r \in R \}$ the set of $R$-centralized
elements.

A covariant functor $F :\Cc \to {\mathcal Set}$ is called
representable if there exists $C\in \Cc$, called the representing
object of $F$, such that $ F \cong \Hom_{\Cc} (C, - )$ in $
{\mathcal Set}\, ^{\Cc}$. ${\bf Rep}\, (\Cc, {\mathcal Set})$ will
be the full subcategory of $ \ul {Functors}\, (\Cc, {\mathcal
Set})$ of all representable functors. The Yoneda lemma states that
for any functor $F: \Cc \to {\mathcal Set}$ and $C\in \Cc$ the map
\begin{equation}\eqlabel{1.3}
\Gamma : \Nat \, (\Hom_{\Cc} (C, -), F) \to F(C), \qquad \Gamma
(\varphi) := \varphi_C ({\rm Id}_C)
\end{equation}
is a bijection between sets with the inverse given by
\begin{equation}\eqlabel{1.4}
\Gamma^{-1} (x)_D (f) := F(f)(x)
\end{equation}
for all $x \in F(C)$, $D\in \Cc$ and $f \in \Hom_{\Cc} (C, D)$. As
a consequence, the functor
\begin{equation}\eqlabel{1.5}
Y : \Cc^{{\rm op}} \to \ul {Functors}\, (\Cc, {\mathcal Set}),
\quad Y(C):= \Hom_{\Cc} (C, - ), \quad Y(f):= \Hom_{\Cc} (f, - )
\end{equation}
for all $C$, $D\in \Cc$ and $f\in  \Hom_{\Cc} (C, D)$ is faithful
and full. Thus, there exists an equivalence of categories
$$
\Cc^{{\rm op}} \cong  {\bf Rep}\, (\Cc, {\mathcal Set}), \qquad C
\mapsto \Hom_{\Cc} (C, - )
$$

Let $\mathcal{V}$ be a variety of algebras in the sense of
universal algebra (for example $\mathcal{V}$ can be the category
of semigroups, monoids, groups, abelian groups, rings, algebras
over commutative rings or modules over a rings, etc.). We recall
from \cite{bergman} the following:

\begin{definition} Let $\mathcal{V}$ be a variety of algebras
and $\gamma : \mathcal{V} \to {\mathcal Set}$ the forgetful
functor. A functor $F: \mathcal{C} \to  \mathcal{V} $ is called
representable if $ \gamma \circ F : \mathcal{C} \to {\mathcal
Set}$ is representable in the classical sense.
\end{definition}

Let $F: \Cc \to \Dd$, $G: \Dd \to \Cc$ be two functors. $F$ is a
left adjoint of $G$ and we denote this by $ F \dashv G$ if there
exist two natural transformations $\eta:\ 1_\Cc\to GF$ and
$\varepsilon: \ FG\to 1_\Dd$, called the unit and counit of the
adjunction, such that
\begin{equation}\eqlabel{1.6}
G(\varepsilon_D)\circ\eta_{G(D)}= I_{G(D)}~~~~{\rm and}~~~~
\varepsilon_{F(C)}\circ F(\eta_C)= I_{F(C)}
\end{equation}
for all $C\in \Cc$ and $D\in \Dd$.

A functor $F: \Cc \to \Dd$ is called a Frobenius functor if there
exists a functor $G$ that is a left and right adjoint of $F$. Let
$ F \dashv G$ be an adjoint pair. Then $F$ is a separable functor
if and only if $\eta:\ 1_\Cc\to GF$ splits: i.e. there exists a
natural transformation $\nu :\ GF \to 1_\Cc$ such that $\nu_C
\circ \eta_C = {\rm Id}_C$, for all $C \in \Cc$. Moreover, $G$ is
separable if and only if $\varepsilon:\ FG\to 1_\Dd$ cosplits,
i.e. there exists a natural transformation $\xi:\ 1_\Dd \to FG$
such that $\varepsilon_D \circ \xi_D = {\rm Id}_S$, for all $D \in
\Dd$. For details and more examples of Frobenius or separable
functors we refer to \cite{BrzezinskiWisbauer}, \cite{CMZ2002}.

Let $R$ be a ring and $C = (C, \Delta, \varepsilon)$ an
$R$-coring: i.e. $C$ is a comonoid in the monoidal category of
$R$-bimodules $ ({}_R{\mathcal M}_{R}, -\otimes_R - , R )$. We
denote by ${\mathcal M}_{R}^C$, ${}_{R}^C{\mathcal M}$ and
${}_{R}^C{\mathcal M}_{R}^C$ the categories of right, left,
respectively $C$-bicomodules. ${\rm Hom}_R^C (M, N)$, ${}_R^C {\rm
Hom} (M, N)$ and ${}_R^C {\rm Hom}_R^C (M, N)$ will be the set of
all morphisms in the categories of right, left and respectively
$C$-bicomodules, for two $C$-comodules $M$ and $N$. A right
$C$-coaction will be denoted by
$$
\rho : M \to M \otimes_R C, \quad \rho (m) = m_{<0>} \otimes_R
m_{<1>}
$$
for all $M \in {\mathcal M}_{R}^C$ and $m\in M$ and a left
$C$-coaction will be denoted by
$$
\rho : M \to C \otimes_R M, \quad \rho (m) = m_{<-1>} \otimes_R
m_{<0>}
$$
for all $M \in {}_{R}^C{\mathcal M}$ and $m\in M$ (summation
understood). The categories ${\mathcal M}_{R}^C$,
${}_{R}^C{\mathcal M}$ and ${}_{R}^C{\mathcal M}_{R}^C$ are
additive and cocomplete (they have all coproducts and coequalizers
\cite[Proposition 18.13]{BrzezinskiWisbauer}).

Let $R$, $S$ be two rings, $C$ an $R$-coring and
${}_{R}^C{\mathcal M}_S$ be the category of all pairs $(V,
\rho_V)$, where $V$ is an $(R,S)$-bimodule, $\rho_V : V \to C\ot_R
V$ is a morphism of $(R,S)$-bimodules and a left $C$-coaction on
$V$. For two objects $V$, $W \in {}_{R}^C{\mathcal M}_S$ we denote
by ${}_R^C {\rm Hom}_S (V, W)$ the set of morphisms in the
category ${}_{R}^C{\mathcal M}_S$, i.e. the set of all
$(R,S)$-bimodule maps $f : V \to W$ that are also left
$C$-comodule maps. The category ${}_{S}\mathcal{M}^C_R$ is defined
similarly.

Let $V \in {}_{R}^C{\mathcal M}_S$. Then we have two functors
$$
V\otimes _S - : {}_{S}{\mathcal M} \to {}_{R}^C{\mathcal M}, \quad
{}_R^C {\rm Hom} (V, -): {}_{R}^C{\mathcal M} \to {}_{S}{\mathcal
M}
$$
where $V \ot_S N \in {}_{R}^C{\mathcal M}$ via $r \cdot (v \ot_S
n) := rv \ot_S n$, $\rho (v \ot_S n) := v_{<-1>} \ot_R v_{<0>}
\ot_S n$, for all $N \in {}_{S}{\mathcal M}$, $r\in R$, $v\in V$
and $n \in N$ and ${}_R^C {\rm Hom} (V, M) \in {}_{S}{\mathcal M}$
via $(s \cdot f) (v) := f (vs)$, for all $M \in {}_{R}^C{\mathcal
M}$, $s\in S$, $f\in {}_R^C {\rm Hom} (V, M)$ and $v\in V$.

The following is the left version of \cite[Theorem 3.2]{Joost} as
a generalization of the Eilenberg–-Watts theorem for categories of
modules.

\begin{theorem}\thlabel{3.1}
Let $R$, $S$ be two rings, $C$ an $R$-coring, $F: {}_{S}{\mathcal
M} \to {}_{R}^C{\mathcal M}$ and $G : {}_{R}^C{\mathcal M} \to
{}_{S}{\mathcal M}$ two functors. Then $F$ is a left adjoint of
$G$ if and only if there exists $V \in {}_{R}^C{\mathcal M}_S$,
unique up to an isomorphism in ${}_{R}^C{\mathcal M}_S$, such that
$$
F \cong V\otimes _S -  \qquad G\cong {}_R^C {\rm Hom} (V, -)
$$
\end{theorem}

Let $U\in {}_{S}\mathcal{M}_R$ and $V \in {}_{R}^C\mathcal{M}_S$.
Then $U\ot_R C \in {}_{S}\mathcal{M}^C_R$ via the right
$C$-coaction
$$
u\ot_R c \mapsto u\ot_R c_{(1)} \ot_R c_{(2)}
$$
for all $u\in U$, $c\in C$ and $V\ot_S U \ot_R C \in
{}_{R}^C\mathcal{M}_R^C$, where the left and the right
$C$-coactions are defined by
$$
v\ot_S u \ot_R c \mapsto v_{<-1>} \otimes_R v_{<0>} \ot_S u \ot_R
c, \qquad v\ot_S u \ot_R c \mapsto v\ot_S u \ot_R c_{(1)} \ot_R
c_{(2)}
$$
for all $v\in V$, $u\in U$ and $c\in C$. Moreover, ${}_R^C {\rm
Hom} (V, C) \in {}_{S}\mathcal{M}_R$, where the right $R$-action
is given by
$$
(f\cdot r) (v) := f(v) r
$$
for all $f \in {}_R^C {\rm Hom} (V, C)$, $r\in R$, $v\in V$.

\section{Computing natural transformations and Yoneda type
embeddings}\selabel{2}

In this section we shall prove all technical results that we shall
use later on. Let $R$, $S$ be two rings, $C$ an $R$-coring, $V$,
$W \in {}_{R}^C{\mathcal M}_S$ and $f : V \to W$ a morphism in
${}_{R}^C{\mathcal M}_S$. We associate to $f$ two natural
transformations:
$$
f\otimes_S -  : V\ot_S - \to W\ot_S -, \qquad v\ot_S n \mapsto f
(v) \ot_S n
$$
for all $N \in {}_{S}{\mathcal M}$, $n\in N$, $v \in V$ and
$$
{}_R^C {\rm Hom} (f, -) : {}_R^C {\rm Hom} (W, -) \to {}_R^C {\rm
Hom} (V, -), \quad \alpha \mapsto \alpha \circ f
$$
for all $M \in {}_{R}^C{\mathcal M}$, $\alpha \in {}_R^C {\rm Hom}
(W, M)$.

\begin{proposition}\prlabel{2.1}
Let $R$, $S$ be two rings, $C$ an $R$-coring. Then:
\begin{enumerate}
\item The functor
$$
Y_1 : {}_{R}^C{\mathcal M}_S \to \ul {Functors}\, \bigl(
{}_{S}{\mathcal M}, \, {}_{R}^C{\mathcal M} \bigl), \qquad Y_1
(V): = V\ot_S -
$$
for all $V \in {}_{R}^C\mathcal{M}_S$ is faithful and
full.\vspace{2mm}

\item The functor
$$
Y_2 :  ({}_{R}^C\mathcal{M}_S )^{{\rm op}} \to \ul {Functors}\,
\bigl( {}_{R}^C{\mathcal M}, \, {}_{S}{\mathcal M} \bigl), \qquad
Y_2 (V): = {}_R^C {\rm Hom} (V, -)
$$
for all $V \in {}_{R}^C\mathcal{M}_S$ is faithful and full.
\end{enumerate}
\end{proposition}

\begin{proof}
1. Let $V$, $W \in {}_{R}^C\mathcal{M}_S$. We have to prove that
$$
(Y_1)_{V, W} : {}_R^C {\rm Hom}_S (V, W) \to \Nat \bigl( V\ot_S -
, \, W\ot_S - \bigl), \quad (Y_1)_{V, W} (f) = f\ot_S -
$$
for all $f\in {}_R^C {\rm Hom}_S (V, W)$ is a bijection between
sets.

Let $\phi : V \ot_S -  \to W \ot_S - $ be a natural
transformation. In particular, $\phi_S : V \ot_S S  \to W \ot_S S$
is a morphism in ${}_{R}^C\mathcal{M}$. We define $f: V \to W$ by
the formula $f := can' \circ \phi_S \circ can $, where $can : V
\to V\ot_S S$ and $can' : W\ot_S S \to W$ are canonical
isomorphisms. Of course $f$ is a morphism in
${}_{R}^C\mathcal{M}$. Using the fact that $\phi$ is a natural
transformation we shall prove that $f$ is also a right $S$-module
map, hence a morphism in ${}_{R}^C\mathcal{M}_S$ and $\phi$ is
uniquely determined by $f$ with the formula $\phi = (Y_1)_{V, W}
(f)$.

Let $N \in {}_{S}\mathcal{M}$ and $n\in N$. Then $u_n : S \to N$,
$u_n (s) := sn$ is a morphism in ${}_{S}\mathcal{M}$. Thus the
diagram
$$
\begin {CD}
V\otimes_S S @> \phi_S >> W\otimes_S S\\
@V Id \otimes_S u_n VV @VV Id\otimes_S u_n V\\
V\otimes_S N @> \phi_N >> W\otimes_S N
\end{CD}
$$
is commutative. We evaluate at $v\ot_S 1_S$ and we obtain that
$$
\phi_N (v \ot_S n) = f(v) \ot_S n
$$
for all $N \in {}_{S}\mathcal{M}$, $v\in V$ and $n\in N$. In
particular, for $N := S$ we obtain that $f$ is also a right
$S$-module map and the above formula tells us that $ \phi =
(Y_1)_{V, W} (f)$.

2. Let $V$, $W \in {}_{R}^C\mathcal{M}_S$. We have to prove that
$$
(Y_2)_{V, W} : {}_R^C {\rm Hom}_S (V, W) \to \Nat \bigl( {}_R^C
{\rm Hom} (W, -)  , \,
 {}_R^C {\rm Hom} (V, -)  \bigl), \, (Y_2)_{V, W} (f) =  {}_R^C {\rm Hom} (f, -)
$$
for all $f\in {}_R^C {\rm Hom}_S (V, W)$ is a bijection between
sets with the inverse given by
$$
(Y_{2}^{-1} )_{V, W} (\theta) = \theta_W ( Id_W ) : V \to W
$$
for any natural transformation $\theta : {}_R^C {\rm Hom} (W, -)
\to {}_R^C {\rm Hom} (V, -)$. This follows straightforward from
the Yoneda lemma if we replace the category ${\mathcal Set}$ with
the category of left $S$-modules. The only two things we have to
prove are that the maps \equref{1.3}, \equref{1.4} from the Yoneda
lemma work properly. More precisely, we note that if $M \in
{}_{R}^C\mathcal{M}$ and $f\in {}_R^C {\rm Hom}_S (V, W)$, then we
can easily show that $Y_2 (f)_M : {}_R^C {\rm Hom} (W, M) \to
{}_R^C {\rm Hom} (V, M)$, $Y_2 (f)_M (\alpha) = \alpha \circ f$,
for all $\alpha \in {}_R^C {\rm Hom} (W, M)$ is a morphism of left
$S$-modules. Finally, if $\theta : {}_R^C {\rm Hom} (W, -) \to
{}_R^C {\rm Hom} (V, -)$ is a natural transformation we have to
prove that $\theta_W (Id_W) : V \to W$ is also a right $S$-module
map, hence a morphism in ${}_{R}^C\mathcal{M}_S$. We shall use
that $\theta$ is a natural transformation. Let $s \in S$ and
$\gamma_s : W \to W$, $\gamma_s (w) := ws$, for all $w \in W$.
Then $\gamma_s$ is a morphism in ${}_{R}^C\mathcal{M}$, thus we
have a commutative diagram
$$
\begin {CD}
{}_R^C {\rm Hom} (W, W)  @> \theta_W >> {}_R^C {\rm Hom} (V, W)\\
@V {}_R^C {\rm Hom} (W, \gamma_s) VV @ VV {}_R^C {\rm Hom}(v ,\gamma_s) V\\
{}_R^C {\rm Hom} (W, W) @> \theta_W >> {}_R^C {\rm Hom} (V, W)
\end{CD}
$$
Now, if we evaluate the diagram at ${\rm Id}_W$ we obtain that
$\theta_W (Id_W)$ is also a right $S$-module map and the proof is
finished.
\end{proof}

In the next two Lemmas we shall compute all natural
transformations between an induction functor and the identity
functor on the category ${}_{R}^C\mathcal{M}$ of left
$C$-comodules. For any object $Z\in {}_{R}^C\mathcal{M}_R$ we
denote by
$$
{}_R^{\bullet}{\rm Hom}_R (Z\ot_R C, \, R)
$$
the set of all $R$-bimodule maps $h: Z\ot_R C \to R$ satisfying the
compatibility condition
\begin{equation}\eqlabel{100}
z_{<-1>} h(z_{<0>} \ot_R c) = h(z \ot_R c_{(1)} ) c_{(2)}
\end{equation}
for all $z\in Z$ and $c\in C$.

\begin{lemma}\lelabel{2.5}
Let $R$ be a ring, $C$ an $R$-coring, $Z\in {}_{R}^C\mathcal{M}_R$
and the induction functor $Z\ot_R - : \, {}_{R}^C\mathcal{M} \to
{}_{R}^C\mathcal{M}$. Then there exists a bijection between sets
$$
\Nat \bigl( Z\ot_R - , \, 1_{{}_{R}^C\mathcal{M}} \bigl) \, \cong
\,  {}_R^C{\rm Hom}_R^C (Z\ot_R C, \, C) \, \cong \,
{}_R^{\bullet}{\rm Hom}_R (Z\ot_R C, \, R)
$$

Explicitly, for any natural transformation $\psi : Z\ot_R - \to
1_{{}_{R}^C\mathcal{M}}$ there exists a unique map $h \in
{}_R^{\bullet}{\rm Hom}_R (Z\ot_R C, \, R)$ such that
\begin{equation}\eqlabel{241}
\psi_M : Z\ot_R M \to M, \quad \psi_M (z \ot_R m) = h ( z \ot_R
m_{<-1>})m_{<0>}
\end{equation}
for all $M \in {}_{R}^C\mathcal{M}$, $m\in M$ and $z\in Z$.
\end{lemma}

\begin{proof} The last bijection follows from Hom-tensor type relations
(\cite{BrzezinskiWisbauer}). More precisely, the map
$$
\alpha : {}_R^{\bullet}{\rm Hom}_R (Z\ot_R C , \, R ) \to {}_R^C
{\rm Hom}_R^C (Z\ot_R C , \, C ), \quad \alpha (h) (z\ot_R c) :=
h(z\ot_R c_{(1)} ) c_{(2)}
$$
for all $h\in {}_R^{\bullet}{\rm Hom}_R (Z\ot_R C , \, R )$, $z\in
Z$ and $c\in C$ is bijective with the inverse given by
$$
\alpha^{-1} (f) (z\ot_R c) := \varepsilon_C \bigl( f (z\ot_R c)
\bigl)
$$
for all $f \in {}_R^C {\rm Hom}_R^C (Z\ot_R C , \, C )$, $z\in Z$
and $c\in C$.

Let now $\psi : Z\ot_R - \to 1_{{}_{R}^C\mathcal{M}}$ be a natural
transformation. In particular, $\psi_C : Z\ot_R C \to C$ is a
morphism in ${}_{R}^C\mathcal{M}$. Using that $\psi$ is a natural
transformation we shall prove that $\psi_C$ is a morphism in
${}_{R}^C\mathcal{M}_R^C$ and $\psi$ is uniquely determined by
$\psi_C$. Let $r\in R$ and $f_r : C \to C$, $f_r (c) := cr$, for
all $c\in C$. Then $f_r$ is a morphism in ${}_{R}^C{\mathcal M}$
and hence the diagram
$$
\begin {CD}
Z\ot_R C @> \psi_C >> C \\
@V Id_Z\ot_R f_r VV @VV f_r V\\
Z\ot_R C @> \psi_C >> C
\end{CD}
$$
is commutative, i.e. $\psi_C$ is also a right $R$-module map. Let
$N\in {}_{R}{\mathcal M}$ and $n\in N$. Then $f_n : C \to C\ot_R
N$, $f_n (c) := c\ot_R n$, for all $c\in C$ is a map in
${}_{R}^C{\mathcal M}$. Thus the diagram
$$
\begin {CD}
Z\ot_R C @> \psi_C >> C \\
@V Id_Z \ot_R f_n VV @VV f_n V\\
Z\ot_R C\ot_R N @> \psi_{C\ot_R N} >> C\ot_R N
\end{CD}
$$
is commutative, which means that
$$
\psi_{C\ot_R N} = \psi_C \ot_R {\rm Id}_N
$$
for all $N\in {}_{R}{\mathcal M}$.

Let $(M, \rho) \in {}_{R}^C{\mathcal M}$; then $\rho : M \to
C\ot_R M$ is a morphism in ${}_{R}^C{\mathcal M}$ so the diagram
$$
\begin {CD}
Z\ot_R M @> \psi_M >> M \\
@V Id_Z \ot_R \rho VV @VV \rho V\\
Z\ot_R C\ot_R M @> \psi_{C\ot_R M} >> C\ot_R M
\end{CD}
$$
is commutative. Using that $\psi_{C\ot_R M} = \psi_C \ot_R Id_M$
we obtain if we evaluate the last diagram at $z\ot_R m$:
\begin{equation}\eqlabel{160}
\rho \bigl( \psi_M (z\ot_R m) \bigl) = \psi_C (z\ot_R m{<-1>})
\ot_R m_{<0>}
\end{equation}

In particular, for $ (M, \rho) = (C, \Delta)$ we obtain that
$\psi_C$ is a morphism in ${}_{R}^C\mathcal{M}_R^C$ and if we
apply $\varepsilon_C$ to the first position in \equref{160} we get
\begin{equation}\eqlabel{161}
\psi_M (z\ot_R m) = \varepsilon_C \bigl(\psi_C (z\ot_R m{<-1>})
\bigl) m_{<0>}
\end{equation}
for all $M \in {}_{R}^C\mathcal{M}$, $z\in C$, $m\in M$, i.e. the
first bijection from the statement.

Now, from the first part of the proof, for any $\psi_C \in
{}_R^C{\rm Hom}_R^C (Z\ot_R C, \, C)$ there exists a unique map
$h\in {}_R^{\bullet}{\rm Hom}_R (Z\ot_R C, \, R)$ such that
$\psi_C (z\ot_R c) = h(z\ot_R c_{(1)} ) c_{(2)}$, for all $z\in Z$
and $c\in C$. Using this formula for $\psi_C$, the equation
\equref{161} takes the form \equref{241} and the proof is
finished.
\end{proof}

\begin{lemma}\lelabel{2.4}
Let $R$ be a ring, $C$ an $R$-coring, $Z\in {}_{R}^C\mathcal{M}_R$
and the induction functor $Z\ot_R - : \, {}_{R}^C\mathcal{M} \to
{}_{R}^C\mathcal{M}$. Then there exists a bijection between sets
$$
\Nat \bigl( 1_{{}_{R}^C\mathcal{M}} , \, Z\ot_R -  \bigl) \, \cong
\,  {}_R^C{\rm Hom}_R^C (C, \, Z\ot_R C) \, \cong \, {}_R^C{\rm
Hom}_R (C, \, Z)
$$
Explicitly, for any natural transformation $\theta :
1_{{}_{R}^C\mathcal{M}} \to Z\ot_R -$ there exists a unique map $p
\in {}_R^C{\rm Hom}_R (C, \, Z)$ such that
\begin{equation}\eqlabel{24}
\theta_M : M \to Z\ot_R M, \quad \theta_M (m) = p (m_{<-1>}) \ot_R
m_{<0>}
\end{equation}
for all $M \in {}_{R}^C\mathcal{M}$ and $m\in M$.
\end{lemma}

\begin{proof} The proof is analogous to \leref{2.5}. The second bijection
is given by Hom-tensor type relations (\cite{BrzezinskiWisbauer}):
the map
\begin{equation}\eqlabel{261}
\beta : {}_R ^C{\rm Hom}_R (C , \, Z ) \to {}_R ^C{\rm Hom}_R^C (C
, \, Z\ot_R C ), \quad \beta (p) (c) :=  p (c_{(1)}) \ot_R c_{(2)}
\end{equation}
for all $p \in {}_R ^C{\rm Hom}_R (C , \, Z )$, $c\in C$ is
bijective with the inverse
$$
\beta^{-1} (g) := (Id_Z \ot_R \varepsilon_C)\circ g
$$
for all $g\in {}_R ^C{\rm Hom}_R^C (C , \, Z\ot_R C )$. Thus any
$\theta_C \in  {}_R^C{\rm Hom}_R^C (C, \, Z\ot_R C)$ has the form
$\theta_C (c) = p (c_{(1)}) \ot_R c_{(2)}$, for a unique $p \in
{}_R ^C{\rm Hom}_R (C , \, Z )$.

Now, let $\theta : 1_{{}_{R}^C\mathcal{M}} \to Z\ot_R - $ be a
natural transformation. Using exactly the same steps from the
proof of \leref{2.5} we can prove that $\theta_C : C \to Z\ot_R C$
is in fact a morphism in ${}_{R}^C\mathcal{M}_R^C$ and $\theta$ is
uniquely determined by $\theta_C$. Details are left to the reader.
\end{proof}

\begin{corollary}\colabel{2.6}
Let $R$ be a ring, $C$ an $R$-coring and $Z\in
{}_{R}^C\mathcal{M}_R$. The following are equivalent:
\begin{enumerate}
\item The induction functor $Z\otimes_R - : {}_{R}^C\mathcal{M}
\to {}_{R}^C\mathcal{M}$ is isomorphic to the identity functor
$1_{{}_{R}^C\mathcal{M}}$ of the category ${}_{R}^C\mathcal{M}$;

\item There exists an isomorphism $Z\ot_R C \cong C$ in the
category ${}_{R}^C\mathcal{M}_R^C$ of $C$-bicomodules;

\item There exists a pair $(p, h)$, where $p \in {}_R^C{\rm Hom}_R
(C, \, Z)$, $h\in {}_R^{\bullet}{\rm Hom}_R (Z\ot_R C, \, R)$ such
that
\begin{equation}\eqlabel{334}
h\bigl( p(c_{(1)}) \ot_R c_{(2)} \bigl) = \varepsilon (c), \quad
h( z \ot_R c_{(1)}) p (c_{(2)}) = z \, \varepsilon (c)
\end{equation}
for all $c\in C$, $z\in Z$.
\end{enumerate}
\end{corollary}

\begin{proof}
$(1) \Rightarrow (2)$ Let $\theta : 1_{{}_{R}^C\mathcal{M}} \to
Z\ot_R - $ and $\psi : Z\ot_R - \to 1_{{}_{R}^C\mathcal{M}}$ be a
pair of natural transformations inverse each to other. From the
proof of \leref{2.4} and \leref{2.5}, $\theta_C$ and $\psi_C$ are
isomorphisms inverse to each other between $C$ and $Z\ot_R C$ in
the category ${}_{R}^C\mathcal{M}_R^C$.

$(2) \Rightarrow (3)$ The pair of maps $(p, h)$ satisfying
\equref{334} parameterizes the isomorphisms between $C$ and
$Z\ot_R C$ in the category ${}_{R}^C\mathcal{M}_R^C$ using
Hom-tensor type relations from the proofs of \leref{2.4} and
\leref{2.5}.

$(3) \Rightarrow (1)$ Let $(p, h)$ be such a pair of maps. Then
$$\theta_M : M \to Z\ot_R M, \quad \theta_M (m) = p (m_{<-1>})
\ot_R m_{<0>}
$$
for all $M \in {}_{R}^C\mathcal{M}$ and $m\in M$ is a natural
isomorphism between the functors $1_{{}_{R}^C\mathcal{M}}$ and
$Z\ot_R - $ with the inverse
$$
\psi_M : Z\ot_R M \to M, \quad \psi_M (z \ot_R m) = h ( z \ot_R
m_{<-1>})m_{<0>}
$$
for all $M \in {}_{R}^C\mathcal{M}$, $m\in M$ and $z\in Z$.
\end{proof}

\section{Representable functors for corings.
Applications}\selabel{3}

In this section we shall use all technical results proven before
in order to obtain the main theorems of the paper. First we shall
give an answer to Problem A:

\begin{theorem}\thlabel{3.2}
Let $R$, $S$ be rings, $C$ an $R$-coring and ${\bf Rep}\, (
{}_{R}^C{\mathcal M}, {}_{S}{\mathcal M} )$ be the category of all
representable functors ${}_{R}^C{\mathcal M} \to {}_{S}{\mathcal
M}$. Then the functor
$$
Y : ({}_{R}^C\mathcal{M}_S )^{{\rm op}} \to {\bf Rep} \,
({}_{R}^C{\mathcal M}, {}_{S}{\mathcal M} ), \quad Y(V):=
{}_R^C{\rm Hom} (V, - )
$$
is an equivalence of categories.
\end{theorem}

\begin{proof}
It follows from $(2)$ of \prref{2.1} that $Y$ is a faithful and
full functor. Let $G \in {\bf Rep}\, ( {}_{R}^C{\mathcal M},
{}_{S}{\mathcal M} )$ be a representable functor.
${}_{R}^C{\mathcal M}$ is a cocomplete category, as it has all
coproducts and coequalizers \cite[Proposition
18.13]{BrzezinskiWisbauer}; thus we can apply Freyd's theorem
\cite[Theorem 8.14]{bergman} to obtain that $G$ is a right
adjoint. Using \thref{3.1} we get that $G \cong {}_R^C {\rm Hom}
(V, -) = Y(V)$, for some $V \in {}_{R}^C{\mathcal M}_S$, i.e. $Y$
is surjective on objects. Thus $Y$ is an equivalence of
categories.
\end{proof}

The following question seems to be hopeless: \textit{Let $R$, $S$
be rings, $C$ an $R$-coring. Describe the category } ${\bf Rep}\,
({}_{S}{\mathcal M}, \, {}_{R}^C{\mathcal M} )$ \textit{of all
representable functors} ${}_{S}{\mathcal M} \to {}_{R}^C{\mathcal
M}$.

We shall indicate now an answer for Problem C:

\begin{corollary}\colabel{proC}
Let $R$, $S$, $T$ be rings, $C$ an $R$-coring and $F :
{}_{R}^C{\mathcal M} \to {}_{S}{\mathcal M}$, $G : {}_{S}{\mathcal
M} \to {}_{T}{\mathcal M}$ representable functors. Then $G\circ F
: {}_{R}^C{\mathcal M} \to {}_{T}{\mathcal M}$ is a representable
functor.
\end{corollary}

\begin{proof} \thref{3.2} gives that there exists
$V \in {}_{R}^C{\mathcal M}_S$ such that $ F \cong {}_R^C{\rm Hom}
(V, - )$. We apply once again \thref{3.2} for the trivial coring
$C = R$ and we obtain that there exists $W \in {}_{S}{\mathcal
M}_T$ such that $G \cong {}_S{\rm Hom} (W, - )$. Now the proof
follows from \thref{3.2} taking into account that there exists a
natural isomorphism of functors given by the Hom-tensor adjunction
$$
{}_S{\rm Hom} (W, - ) \circ {}_R^C{\rm Hom} (V, - ) \cong
{}_R^C{\rm Hom} (V\otimes_S W,  - )
$$
where $V\otimes_S W  \in {}_{R}^C{\mathcal M}_T$ with the left
$C$-coaction implemented by the coaction on $V$.
\end{proof}

\textbf{The induction functor. }Let $R$, $S$ be rings, $C$ an
$R$-coring and $U \in {}_{S}\mathcal{M}_R$. In the last part of
the paper we shall give necessary and sufficient conditions for
the induction functor
$$
U\otimes_R - : {}_{R}^C\mathcal{M} \to
{}_{S}\mathcal{M}
$$
to be: a representable functor, an equivalence of categories, a
separable or a Frobenius functor.

\begin{example}\exlabel{exem} Let us give the motivation for the first
problem. Consider  $U : = S = R$. Then $R\otimes_R - \cong F$,
where $F : {}_{R}^C\mathcal{M} \to {}_{R}\mathcal{M}$ is the
forgetful functor. We have the adjoint pairs of functors:
$$
F = R\otimes_R - \, \dashv C\otimes_R - \, \dashv {}_R^C {\rm Hom}
(C, -)
$$
Now \cite[Theorem 4.1]{Brz:cor} gives three necessary and
sufficient conditions for the forgetful functor $F \cong
R\otimes_R -$ to be a Frobenius functor. This can be restated as
$F \cong R\otimes_R -$ is a representable functor having $C$ as an
object of representability. In the following we shall address the
general case of an arbitrary induction functor; moreover we shall
not impose restrictive conditions regarding the object of
representability.
\end{example}

First we shall give necessary and sufficient conditions for $
U\otimes_R - $ to be an equivalence of categories. Morita type
theorems for categories of comodules over corings where also
proved in \cite{bv}. The next theorem is not a special case of
them. The proof we give is elementary, being based \thref{3.1} and
\coref{2.6}.

\begin{theorem}\thlabel{3.3}
Let $C$ be an $R$-coring and $U \in {}_{S}\mathcal{M}_R$. The
following are equivalent:
\begin{enumerate}
\item $U\otimes_R - : {}_{R}^C\mathcal{M} \to {}_{S}\mathcal{M}$
is an equivalence of categories;

\item There exists $V \in {}_{R}^C{\mathcal M}_S$ such that:
\begin{enumerate}
\item[(i)] $U\otimes_R V \cong S$, isomorphism in ${}_{S}{\mathcal
M}_{S}$.

\item[(ii)] $V\otimes_S U \otimes_R C \cong C$, isomorphism in
${}_{R}^C{\mathcal M}_{R}^C$;
\end{enumerate}

\item There exists a triple $(V, p, h)$, where $V \in
{}_{R}^C{\mathcal M}_S$, $p \in {}_R^C {\rm Hom}_R (C, V\otimes_S
U)$, $h \in {}_R{\rm Hom}_R (V\otimes_S U\otimes_R C, R)$ such
that:
\begin{enumerate}
\item[(i)] $U\otimes_R V \cong S$, isomorphism in ${}_{S}{\mathcal
M}_{S}$;

\item[(ii)] $v_{<-1>} h ( v_{<0>} \otimes_S u \otimes_R c ) = h (v
\otimes_S u \otimes_R c_{(1)} ) c_{(2)}$;

\item[(iii)] $h(p (c_{(1)}) \otimes_R c_{(2)}) = \varepsilon (c)$;

\item[(iv)] $h (v \otimes_S u \otimes_R c_{(1)} ) p(c_{(2)}) = v
\otimes_S u \varepsilon(c)$
\end{enumerate}
for all $v\in V$, $u\in U$, $c\in C$.
\end{enumerate}
\end{theorem}

\begin{proof} $ (1) \Leftrightarrow (2) $ First we note that,
if the functor $U\otimes_R - : {}_{R}^C\mathcal{M} \to
{}_{S}\mathcal{M}$ is an equivalence of categories, then its
inverse $F$ is a left (and a right) adjoint. Using \thref{3.1} we
obtain that there exists $V \in {}_{R}^C{\mathcal M}_S$, unique up
to an isomorphism in ${}_{R}^C{\mathcal M}_S$, such that $F \cong
V\otimes _S -$. Thus $(1)$ can be restated as: $(U\otimes_R - ,
V\otimes _S -)$ is an equivalence of categories inverse each
other.

Let $F := V\ot_S -$ and $G:= U\ot_R -$. Then $F\circ G = V\ot_S U
\ot_R - = Z \ot_R -$, where $Z:= V\ot_S U \in {}_{R}^C{\mathcal
M}_{R}$. Using \coref{2.6} we obtain that $F\circ G \cong
1_{{}_{R}^C\mathcal{M}}$ if and only if $(ii)$ of $(2)$ holds. On
the other hand, $G\circ F = U\ot_R V \ot_S - = T\ot_S -$, where $T
:= U\ot_R V \in {}_{S}{\mathcal M}_{S}$. Thus, $G\circ F \cong
1_{{}_{S}\mathcal{M}} = S\ot_S - $ if and only if $U\ot_R V \cong
S$, isomorphism in ${}_{S}{\mathcal M}_{S}$.

$ (2) \Leftrightarrow (3) $ The pair of maps $(p, h)$ and the
conditions $(ii) - (iv)$ of $(3)$ give the parametrization of
isomorphisms in ${}_{R}^C{\mathcal M}_{R}^C$ between $V\otimes_S U
\otimes_R C$ and $C$ according to \coref{2.6} applied for $Z: =
V\ot_S U \in {}_{R}^C{\mathcal M}_{R}$. The condition $(ii)$ in
$(3)$ expresses the fact that $h \in {}_R^{\bullet}{\rm Hom}_R
(V\ot_S U\ot_R C, \, R)$.
\end{proof}

In order to study the representability of the induction functor
$U\otimes_R -$ we need to introduce the following concept:

\begin{definition}
Let $R$, $S$ be two rings, $C$ an $R$-coring, $U \in
{}_{S}\mathcal{M}_R$ and $V \in {}_{R}^C{\mathcal M}_S$.

A pair $(e, h)$, where $e = \sum e^1 \otimes e^2 \in
\bigl(U\otimes_R V \bigl)^S$, $h \in {}_R{\rm Hom}_R (V\otimes_S
U\otimes_R C, R)$, such that
\begin{equation}\eqlabel{db11}
v_{<-1>} h ( v_{<0>} \otimes_S u \otimes_R c ) = h (v \otimes_S u
\otimes_R c_{(1)} ) c_{(2)}
\end{equation}
\begin{equation}\eqlabel{db12}
\sum e^1 \, h ( e^2 \otimes_S u \otimes_R c ) = u \varepsilon (c)
\end{equation}
\begin{equation}\eqlabel{db13}
\sum h (v \otimes_S e^1 \otimes_R e^2_{<-1>}) e^2_{<0>} = v
\end{equation}
for all $v\in V$, $u\in U$, $c\in C$ is called a \textit{comodule
dual basis of first kind} for $(U, V)$.
\end{definition}

\begin{remarks} 1. We shall look at the module case in order to explain the
terminology. Let $C: = R$, $U \in {}_{S}\mathcal{M}_R$, $V := U^*
= {\rm Hom}_R (U, R) \in {}_{R}\mathcal{M}_S$ and $h$ the
evaluation map
$$
h := ev_U : U^* \otimes_S U \to R, \quad u^*\otimes_S u \mapsto
<u^*, u>
$$
There exists $ e = \sum_i u_i \otimes_R u_i^* \in (U\otimes_R
U^*)^S$ such that $(e, ev_U)$ is a comodule dual basis of first
kind for $(U, U^*)$ if and only if $\{u_i, u_i^* \}$ is a dual
basis for $U \in \mathcal{M}_R$. This is equivalent to $U$ is
finitely generated projective as a right $R$-module.

2. Let $(e, h)$ be a comodule dual basis of first kind for $(U,
V)$. Then $V$ is finitely generated projective as a left
$R$-module: indeed, it follows from \equref{db13} that $\{ h (?
\ot_S e^1 \ot_R e^2_{<-1>}), \, e^2_{<0>} \}$ is a dual basis for
$V$ as a left $R$-module.

3. Using \exref{exem}, \thref{3.4} below and \cite[Theorem
4.1]{Brz:cor} we obtain the following: let $U := S := R$. Then
there exists a comodule dual basis of the first type for $(R, C)$
if and only if $C$ is finitely generated projective as a left
$R$-module and the extension $R \to C^* = {\rm Hom}_R (C, R)$ is a
Frobenius extension of rings in the classical sense.
\end{remarks}

Let $X \in {}_{S}\mathcal{M}_S$. We recall two well know results
(in fact they are also special cases of \leref{2.4} and
\leref{2.5} for the trivial coring $C = R$). For any natural
transformation $\eta : 1_{{}_{S}\mathcal{M}} \to X\ot_S -$ there
exists a unique element $e \in X^S := \{ x \in X \, | \,\, sx =
xs, \forall \, s\in S \}$ such that
$$
\eta_N : N \to X\ot_S N, \quad \eta_N (n) = e \ot_S n
$$
for all $N \in {}_{S}\mathcal{M}$ and $n\in N$ and for any natural
transformation $\varphi: X\ot_S - \to 1_{{}_{S}\mathcal{M}}$ there
exists a unique map $E \in {}_S{\rm Hom}_S (X, S)$ such that
$$
\varphi_N : X\ot_S N \to N, \quad \varphi_N (x \ot_S n) = E(x)n
$$
for all $N \in {}_{S}\mathcal{M}$, $x\in X$ and $n\in N$.

Now we are ready to give an answer to Problem B for an induction
functor $U\otimes_R - : {}_{R}^C\mathcal{M} \to
{}_{S}\mathcal{M}$; \cite[Theorem 2.1]{morita} is recovered for
the trivial coring $C := R$ and \cite[Theorem 4.1]{Brz:cor} is
obtained as special case for $U := S := R$ if we predefine $C$ to
be the object of representability of the induction functor in the
next theorem.

\begin{theorem}\thlabel{3.4}
Let $C$ be an $R$-coring and $U \in
{}_{S}\mathcal{M}_R$. The following are equivalent:

\begin{enumerate}

\item The induction functor $U\otimes_R - : {}_{R}^C\mathcal{M}
\to {}_{S}\mathcal{M}$ is representable;

\item There exists $V \in {}_{R}^C{\mathcal M}_S$ such that $
V\otimes_S -$ is a left adjoint of  $U\otimes_R -$;

\item There exists $(V, e, h)$, where  $V \in {}_{R}^C{\mathcal
M}_S$ and $(e, h)$ is a comodule dual basis of first kind for $(U,
V)$.
\end{enumerate}
In this case $ U\otimes_R - \cong {}_R^C {\rm Hom} (V, -)$ and $V$
is finitely generated and projective as a left $R$-module.
\end{theorem}

\begin{proof} $ (1) \Leftrightarrow (2) $ It follows from
\thref{3.2} that a representable functor ${}_{R}^C\mathcal{M} \to
{}_{S}\mathcal{M}$ is isomorphic to ${}_R^C {\rm Hom} (V, -)$, for
some $V \in {}_{R}^C{\mathcal M}_S$. Now, $V\ot_S - $ is a left
adjoint of ${}_R^C {\rm Hom} (V, -)$; hence the conclusion follows
from Khan's theorem of uniqueness of adjoints.

$ (2) \Leftrightarrow (3)$ Let $V \in {}_{R}^C{\mathcal M}_S$. We
shall prove that $V\otimes_S -$ is a left adjoint of  $U\otimes_R
-$ if and only if there exists $(e, h)$ a comodule dual basis of
first kind for $(U, V)$.

Indeed, for any natural transformation $\eta :
1_{{}_{S}\mathcal{M}} \to U\ot_R V \ot_S -$ there exists a unique
element $e = \sum e^1 \ot_R e^2 \in (U\ot_R V)^S$ such that
\begin{equation}\eqlabel{unit10}
\eta_N : N \to U\ot_R V \ot_S N, \quad \eta_N (n) = \sum e^1\ot_R
e^2 \ot_S n
\end{equation}
for all $N \in {}_{S}\mathcal{M}$ and $n\in N$. On the other hand,
if we apply \leref{2.5} for $Z := V\ot_S U$ we obtain: for any
natural transformation $\varepsilon : V\ot_S U\ot_R - \to
1_{{}_{R}^C\mathcal{M}}$ there exists a unique $h \in
{}_R^{\bullet}{\rm Hom}_R (V\ot_S U\ot_R C, \, R)$ such that
\begin{equation}\eqlabel{unit11}
\varepsilon_M : V\ot_S U \ot_R M \to M, \quad \varepsilon_M
(v\ot_S u \ot_R m) = h ( v\ot_S u \ot_R m_{<-1>})m_{<0>}
\end{equation}
for all $M \in {}_{R}^C\mathcal{M}$, $m\in M$, $v\in V$ and $u\in
U$. We note that \equref{db11} means that $h \in
{}_R^{\bullet}{\rm Hom}_R (V\ot_S U\ot_R C, \, R)$.

We shall prove that the above pair of natural transformations
$(\eta, \varepsilon)$ meets the condition of adjunction
\equref{1.6} if and only if \equref{db12} and \equref{db13} hold.
We denote $G = U\ot_R -$ and $F = V\ot_S -$. By a direct
calculation we have $G(\varepsilon_M)\circ\eta_{G(M)}= Id_{G(M)}$,
for all $M \in {}_{R}^C\mathcal{M}$ if and only if
\begin{equation}\eqlabel{369}
u\ot_R m = \sum e^1 h (e^2 \ot_S u \ot_R m_{<-1>} ) \ot_R m_{<0>}
\end{equation}
for all $M \in {}_{R}^C\mathcal{M}$, $m\in M$ and $u\in U$. Now,
\equref{369} follows from \equref{db12}. Conversely, if we
consider $M := C$ and apply $Id\ot_R \varepsilon$ to \equref{369}
we obtain \equref{db12}.

Finally, $\varepsilon_{F(N)}\circ F(\eta_N)= Id_{F(N)}$, for all
$N \in {}_{S}\mathcal{M}$ if and only if
$$
v \ot_S n = \sum h (v \ot_S e^1 \ot_R e^2_{<-1>} ) e^2_{<0>} \ot_S
n
$$
for all for all $N \in {}_{S}\mathcal{M}$,  $v\in V$, $n\in N$ and
this condition is obviously equivalent to \equref{db13}.
\end{proof}

\begin{corollary}\colabel{proC2}
Let $R$ be a ring, $C$ an $R$-coring. The following are
equivalent:
\begin{enumerate}
\item The forgetful functor $F : {}_{R}^C{\mathcal M} \to
{}_{R}{\mathcal M}$ is representable;

\item There exists $(V, e, h)$, where  $V \in {}_{R}^C{\mathcal
M}_R$, $e\in V^R$ and $h \in {}_R{\rm Hom}_R (V\otimes_R C, R)$,
such that
\begin{equation}\eqlabel{db112}
v_{<-1>} h (v_{<0>} \otimes_R c ) = h (v \otimes_R c_{(1)} )
c_{(2)}
\end{equation}
\begin{equation}\eqlabel{db122}
h ( e \otimes_R c ) = \varepsilon (c)
\end{equation}
\begin{equation}\eqlabel{db132}
h (v \otimes_R e_{<-1>}) e_{<0>} = v
\end{equation}
for all $v\in V$, $c\in C$.

\end{enumerate}
\end{corollary}

\begin{proof}
We apply \thref{3.4} for $U = S = R$. In this case the induction
functor $R\ot_R -$ is isomorphic to the forgetful functor. The
conditions \equref{db112}, \equref{db122}, \equref{db132} mean
that $(e, h)$ is a comodule dual basis of first kind for $(R, V)$.
\end{proof}

\begin{corollary}\colabel{3.5}
Let $R$, $S$ be two rings, $C$ an $R$-coring, $U \in
{}_{S}\mathcal{M}_R$ and $V \in {}_{R}^C{\mathcal M}_S$. Assume
that there exists $ (e, h)$ a comodule dual basis of first kind
for $(U, V)$. Then:

\begin{enumerate}
\item The induction functor $ V \otimes_S - : {}_{S}{\mathcal M}
\to {}_{R}^C{\mathcal M}$ is separable if and only if there exists
$ E \in {}_S {\rm Hom}_S (U\otimes_R V, S) $ such that $E (e) =
1$.

\item The induction functor $U\otimes_R - : {}_{R}^C{\mathcal
M}\to {}_{S}{\mathcal M}$ is separable if and only if there exists
$p \in {}_R^C{\rm Hom}_R (C, V\otimes_S U)$ such that:
$$
h(p (c_{(1)}) \otimes_R c_{(2)}) = \varepsilon (c)
$$
for all $c\in C$.
\end{enumerate}
\end{corollary}

\begin{proof} With our assumptions $V\ot_S - $ is a left adjoint of
$U\ot_R -$ (\thref{3.4}) with the unit and counit given by
\equref{unit10} and \equref{unit11}.

$(1)$ Being a left adjoint, $V\ot_S - $ is a separable functor if
and only if the unit $\eta$ of the adjunction $V\ot_S - \dashv
U\ot_R -$ splits, that is there exists $\nu : U\ot_R \ot_S - \to
1_{{}_S {\mathcal M}}$ a natural transformation such that $\nu_N
\circ \eta_N = Id_N$ for all $N \in {}_S {\mathcal M}$. Such a
natural transformation $\nu$ is uniquely defined by a map $ E \in
{}_S {\rm Hom}_S (U\otimes_R V, S)$ via the formula $\nu_N (u\ot_R
v \ot_S n) = E (u\ot_R v)n$, for all $N \in {}_S {\mathcal M}$,
$u\in U$, $v\in V$ and $n\in N$. It is easy to see that $\nu_N$
splits $\eta_N$ if and only if $E (e) = 1_S$.

$(2)$ $U\ot_R -$ is a right adjoint: hence, it is separable if and
only if the counit $\varepsilon$ of the adjunction $V\ot_S -
\dashv U\ot_R -$ cosplits; that is there exists a natural
transformation $\xi : 1_{{}_{R}^C\mathcal{M}} \to V\ot_S U \ot_R
-$ such that $\varepsilon_M \circ \xi_M = Id_M$ for all $M\in
{}_{R}^C\mathcal{M}$. It follows from \leref{2.4} that such a
natural transformation $\xi$ is uniquely defined by a map $p \in
{}_R^C{\rm Hom}_R (C, \, V\ot_S U)$ such that
$$
\xi_M : M \to V\ot_S U\ot_R M, \quad \xi_M (m) = p (m_{<-1>})
\ot_R m_{<0>}
$$
for all $M \in {}_{R}^C\mathcal{M}$ and $m\in M$. Now, we can
prove directly that $\xi$ cosplits $\varepsilon$ if and only if
$$
h \Bigl( p (m_{<-1> (1)}) \ot_R m_{<-1> (2)} \Bigl) m_{<0>} = m
$$
for all $M \in {}_{R}^C\mathcal{M}$ and $m\in M$. This condition
is obviously equivalent (take $M = C$ and apply the counit of $C$
on the second position, the converse is trivial) to $h(p (c_{(1)})
\otimes_R c_{(2)}) = \varepsilon (c)$ for all $c\in C$.
\end{proof}

\begin{definition}\delabel{db2}
Let $R$, $S$ be two rings, $C$ an $R$-coring, $U \in
{}_{S}\mathcal{M}_R$ and $V \in {}_{R}^C{\mathcal M}_S$. A pair of
maps $(p, E)$, where $p \in {}_R^C {\rm Hom}_R (C, V\otimes_S U)$,
$E \in {}_S{\rm Hom}_S (U\otimes_R V, S)$ such that the following
diagrams
$$
\xymatrix { &V \ar[rr]^{\rho_{V}}\ar[d]_{\wr}
&{}\ar@{..}[r]
&C\otimes_{R}V\ar[d]^{p\otimes_{R}Id_{V}}\\
&V\otimes_{S}S &{}\ar@{..}[l]
&V\otimes_{S}U\otimes_{R}V\ar[ll]^{Id_{V}\otimes_{S}E} } \,
\xymatrix {
&U\otimes_{R}C\ar[rr]^{Id_{U}\otimes_{R}p}\ar[d]_{Id\otimes_{R}\varepsilon}
&{}\ar@{..}[r] &U\otimes_{R}V\otimes_{S}U\ar[d]^{E\otimes_{S}Id_{U}}\\
&U\otimes_{R}R\ar[r]^{\sim} &U\ar[r]^{\sim} &S\otimes_{S}U }
$$
are commutative is called a \textit{comodule dual basis of the
second kind} for $(U, V)$.
\end{definition}

\begin{examples}
1. Let $C: = R$ be the trivial coring, $V \in {}_{R}{\mathcal
M}_S$, $U: = V^* \in {}_{S}{\mathcal M}_R$ its right dual.
Consider the evaluation map
$$
E : V^* \otimes_R V \to S, \quad E (v^* \otimes_R v) = <v^*, v>
$$
Then there exists $p \in {}_R {\rm Hom}_R (R, V\otimes_S V^*)$
such that $(p, E)$ is a comodule dual basis of the second kind for
$(V^*, V)$ if and only if $V$ is finitely generated and projective
as a right $S$-module.

2. Let $U = S = R$ and $V = C$. Then $(Id_C, \varepsilon_C)$ is a
comodule dual basis of the second kind for $(R, C)$.
\end{examples}

The reverse side of the adjunction of the same induction functors
is also interesting:

\begin{theorem}\thlabel{3.6}
Let $R$, $S$ be two rings, $C$ an $R$-coring, $U \in
{}_{S}\mathcal{M}_R$ and $V \in {}_{R}^C{\mathcal M}_S$. The
following are equivalent:

\begin{enumerate}
\item The induction functor $U\otimes_R - : {}_{R}^C{\mathcal
M}\to {}_{S}{\mathcal M}$ is a left adjoint of $ V \otimes_S - :
{}_{S}{\mathcal M} \to {}_{R}^C{\mathcal M}$; \vspace{3mm}

\item There exists $(p, E)$ a comodule dual basis of the second
kind for $(U, V)$.
\end{enumerate}
\end{theorem}

\begin{proof} It follows from \leref{2.4} for $Z = V\ot_S U$ that
a natural transformation $\eta : 1_{{}_{R}^C\mathcal{M}} \to
V\ot_S U \ot_R -$ is uniquely defined by a map $p \in {}_R^C{\rm
Hom}_R (C, \, V\ot_S U)$ such that
\begin{equation}\eqlabel{328}
\eta_M : M \to V\ot_S U\ot_R M, \quad \eta_M (m) = p (m_{<-1>})
\ot_R m_{<0>}
\end{equation}
for all $M \in {}_{R}^C\mathcal{M}$ and $m\in M$.

A natural transformation $\varepsilon : U\ot_R \ot_S - \to 1_{{}_S
{\mathcal M}}$ is uniquely defined by a map $ E \in {}_S {\rm
Hom}_S (U\otimes_R V, S)$ such that
\begin{equation}\eqlabel{329}
\varepsilon_N (u\ot_R v \ot_S n) = E (u\ot_R v)n
\end{equation}
for all $N \in {}_S {\mathcal M}$, $u\in U$, $v\in V$ and $n\in
N$.

Now we shall prove that $(\eta, \varepsilon)$ given by
\equref{328} and \equref{329} fulfill the condition of adjunction
\equref{1.6} if and only if the pair of maps $(p, E)$ that defines
the natural transformations $\eta$ and $\varepsilon$ is a comodule
dual basis of the second kind for $(U, V)$. We denote $F = U\ot_R
-$ and $G = V\ot_S -$ and we shall adopt the notation $p (c) =
\sum p(c)^V \ot p(c)^U \in V\ot_S U$, for all $c\in C$.

By a direct calculation we have $G(\varepsilon_N)\circ\eta_{G(N)}=
Id_{G(N)}$, for all $N \in {}_{S}\mathcal{M}$ if and only if
\begin{equation}\eqlabel{3690}
v\ot_S n = \sum  p(v_{<-1>})^V \, E \Bigl( p(v_{<-1>})^U \ot_R
v_{<0>} \Bigl) \ot_S n
\end{equation}
for all $N \in {}_{S}\mathcal{M}$, $n\in N$ and $v\in V$. Now,
\equref{3690} is equivalent (take $N= S$, $n = 1_S$) to the fact
that the left diagram of \deref{db2} is commutative.

On the other hand $\varepsilon_{F(M)}\circ F(\eta_M)= Id_{F(M)}$,
for all $M \in {}_{R}^C\mathcal{M}$ if and only if
\begin{equation}\eqlabel{3691}
u\ot_R m = \sum  E \Bigl (u \ot_R  p(m_{<-1>})^V  \bigl) \,
p(v_{<-1>})^U  \ot_R m_{<0>}
\end{equation}
for all $M \in {}_{R}^C\mathcal{M}$, $u \in U$ and $m \in M$. Now,
\equref{3691} is equivalent to the fact that the right diagram of
\deref{db2} is commutative. Indeed, if we take $M = C$ and $m =
c\in C$ and apply $\varepsilon_C$ to \equref{3691} we obtain the
commutativity of the diagram. The converse is straightforward.
\end{proof}

The fact that the forgetful functor $F : {}_{R}^C\mathcal{M} \to
{}_{R}\mathcal{M}$ has a right adjoint \cite[Lemma 3.1]{Brz:cor}
is a special case of \thref{3.6} as $(Id_C, \varepsilon_C)$ is a
comodule dual basis of the second kind for $(R, C)$. Moreover, the
following Corollary is a generalization of \cite[Theorem 3.3 and
  Theorem 3.5]{Brz:cor} which are obtained if
we consider $U := S := R$ and $V := C$ taking into account that
$(Id_C, \varepsilon_C)$ is a comodule dual basis of the second
kind for $(R, C)$.

\begin{corollary}\colabel{3.7}
Let $R$, $S$ be two rings, $C$ an $R$-coring, $U \in
{}_{S}\mathcal{M}_R$ and $V \in {}_{R}^C{\mathcal M}_S$. Assume
that there exists $ (p, E)$ a comodule dual basis of the second
kind for $(U, V)$. Then:

\begin{enumerate}
\item The induction functor $ V \otimes_S - : {}_{S}{\mathcal M}
\to {}_{R}^C{\mathcal M}$ is separable if and only if there exists
an element $ e \in (U\otimes_R V)^S$ such that $E (e) = 1$.

\item The induction functor $U\otimes_R - : {}_{R}^C{\mathcal
M}\to {}_{S}{\mathcal M}$ is separable if and only if there exists
$h \in {}_R{\rm Hom}_R (V\otimes_S U\otimes_R C, R)$ s.t.:
$$v_{<-1>} h ( v_{<0>} \otimes_S u \otimes_R c ) = h (v \otimes_S u
\otimes_R c_{(1)} ) c_{(2)}$$
$$
h \bigl( p(c_{(1)}) \otimes_R c_{(2)} \bigl) = \varepsilon(c)
$$
for all $v\in V$, $u\in U$, $c\in C$.
\end{enumerate}
\end{corollary}

\begin{proof}
With our assumptions, $U\ot_R - $ is a left adjoint of $V\ot_S -$
(\thref{3.6}) with the unit and counit given by \equref{328} and
\equref{329}. Using \leref{2.5} the proof follows similarly to the
one of \coref{3.5}.
\end{proof}

\begin{remark}
In general, the separability of an induction functor $ V \otimes_S
- : {}_{S}{\mathcal M} \to {}_{R}^C{\mathcal M}$  is still an open
problem even for the category of modules, i.e. for the trivial
coring $C := R$. \cite[Corollary 5.11]{CaenepeelKadison} solved
the problem only for finitely generated and projective modules,
that is in the case that the induction functor is representable.
All four statements of \coref{3.5} and \coref{3.7} generalize
their result.
\end{remark}

\end{document}